\documentclass[12pt]{amsart}

\usepackage{hyperref}
\usepackage{graphicx}
\usepackage{amsfonts,amsthm,amsmath,amssymb}
\usepackage{enumitem}
\usepackage[table,xcdraw]{xcolor}
\usepackage{tikz}
\usetikzlibrary{calc}

\theoremstyle{definition}
\newtheorem{thm}{Theorem}[section]
\newtheorem{lem}[thm]{Lemma}

\newcommand{\Z}{\mathbb{Z}}
\newcommand{\hmax}{h_{\text{max}}}
\newcommand{\hmin}{h_{\text{min}}}
\newcommand{\hext}{h_{\text{ext}}}

\title{A lower bound on forcing numbers based on height functions}

\author[Fateh Aliyev]{Fateh Aliyev}
\address[Fateh Aliyev]{Venice High School,  Los Angeles, CA 13000.}
\email{\texttt{fateh.gurseloglu@gmail.com}}

\author{Nikita Gladkov}
\address[Nikita Gladkov]{Department of Mathematics, UCLA,  Los Angeles, CA 90095.}
\email{\texttt{gladkovna@ucla.edu}}

\begin{document}

\begin{abstract}
    We establish a lower bound on the forcing numbers of domino tilings computable in polynomial time based on height functions. This lower bound is sharp for a $2n$ by $2n$ square as well as other cases.
\end{abstract}

\maketitle

\section{Introduction}

For any graph $G$ and a perfect matching $M$ of $G$, a \textit{forcing
set} of $M$ is a subset $S$ of $M$ such that $S$ is contained in no other
perfect matching of $G$. The cardinality of the smallest such forcing set is
called the \textit{forcing number} of $M$, denoted by $f(M, G)$ \cite{Hkz}. The
smallest forcing number over all perfect matchings of $G$ is called the
\textit{forcing number} of $G$. 

Let $R$ be a simply connected region in the square lattice $\Z^2$, formed 
by joining one or more equal squares edge to edge so that the corners
of the $1 \times 1$ squares are points in $\Z^2$.
The dual graph of $R$ has vertices at the
centers of the squares. Two vertices are connected if the associated squares
are adjacent. We observe that perfect matchings of this graph correspond to domino
tilings of $R$ with $1 \times 2$ or $2 \times 1$ rectangles. 
For any tiling $T$ of $R$, the forcing number of $T$ is denoted by $f(T,
R)$, and the minimum such forcing number over all tilings of $R$ is called the
forcing number of $R$, denoted by $f(R)$. 

\begin{figure}[h]
\centering
\begin{tikzpicture}

    \draw (0, 0) grid (6, 6);
    \draw[fill=gray] (0, 0) rectangle (2, 1);
    \draw[fill=gray] (1, 1) rectangle (3, 2);
    \draw[fill=gray] (2, 2) rectangle (4, 3);
\end{tikzpicture}
\caption{Minimal forcing set in a $6 \times 6$ square.}
\label{fig:ONEM}
\end{figure}
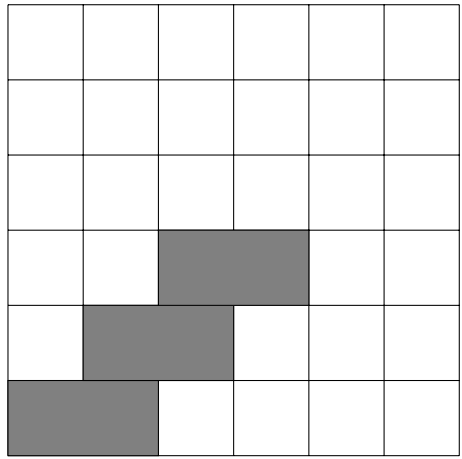

The forcing number of a  $2n$ by $2n$ square was found to be $n$ in \cite{Pk}.
The upper bound on the forcing number is given by a ``staircase'' example 
in Figure~\ref{fig:ONEM}. The lower bound is proved using 
a clever symmetry argument. The present paper gives a new proof of this result 
that generalizes to non-symmetric regions. More recently, the forcing numbers of thin
rectangles (rectangles of dimensions $2 \times n$ and $3 \times n$) have been
found using algebraic techniques \cite{Zz}. A more general technique was
presented in \cite{Rid}, and this technique was applied to compute the
forcing number of so-called ``stop signs'' \cite{Lp}, as well as non-planar
graphs such as the torus and the hypercube \cite{Rid}.

When $R$ is a simply connected closed region with no holes, it is determined
by its boundary $\partial R$. We impose a
black and white coloring on the squares of $R$ in the following way: if the
lower-left corner of a square corresponds to the point $(a, b)$ such that $a +
b \equiv 0 \pmod 2$, color it white; otherwise, color it black. We say that 
an edge $(x, y)$ of the grid belongs to a tiling $T$ if 
it is a part of a domino's boundary. A result proven
in \cite{Thu} (see also \cite{Fou}) states that for any tiling $T$ of $R$ with
dominoes, there exists a \textit{height function} $h_T: R \to \Z$, 
unique up to an integer constant such that for every edge $(x, y)$ in
$T$ such that there is a white square to the left, $h_T(x) - h_T(y) = 1$. 

Furthermore, Thurston showed in the same paper that if the value of 
the height function on $\partial R$ is fixed, 
then there exist a maximum height function $\hmax$ and a minimum height function $\hmin$
such that every other possible height function $h$ satisfies
$\hmin(x) \le h(x) \le \hmax(x)$
for all points $x$ in $R$. With this information, he was able
to provide an algorithm which could then calculate $\hmax$ or $\hmin$ in 
$O(n\log n)$ time, where $n = |R|$ \cite{Thu}. 
In this paper we use $\hmin$ and $\hmax$ in the lower bound on 
the forcing number of a region $R$. While our bound is weaker than 
the one by Riddle, it is much easier to compute.

\begin{thm}\label{main}
The forcing number of a region $R$ on the square lattice
is bounded below by 
$$\frac{1}{4} \max_{x \in R}\big(\hmax(x) - \hmin(x)\big).$$
\end{thm}

We also derive a similar bound for the triangular lattice.
\begin{thm}\label{thm:trimain}
The forcing number of a region $R$ on the triangular lattice
is bounded below by 
$$\frac{1}{3} \max_{x \in R}\big(\hmax(x) - \hmin(x)\big).$$
\end{thm}

Finally, in Theorem~\ref{thm:forcehex} we apply this bound to compute the forcing numbers of hexagons. To the best of our knowledge, they were never computed before.

\section{Characterizing the difference}

The following lemma by Fournier characterizes all height functions on a region $R$, and
therefore all tilings of $R$:

\begin{lem}[\cite{Fou}]\label{fou}
Let $H_R$ be the set of all height functions that correspond to a tiling of
$R$. Then, a function $h: R \to \Z$ belongs to $H_R$ if and only if the
following conditions hold:
\begin{enumerate}
\item For every $x, y \in R$ such that $x = (a_1, b_1)$, $y = (a_2, b_2)$, and
	$a_1 - a_2 = b_1 - b_2 = 0 \pmod 2$, we have $h(x) \equiv h(y) \pmod 4$.
\item For every edge $(x, y) \in \partial R$ such that when crossing from $x$
	to $y$ there is a white square to our left, $h(x) - h(y) = 1$.
\item For every edge $(x, y) \in R$, $|h(x) - h(y)| \le 3$.
\end{enumerate}
\end{lem}

Following \cite{Pst}, consider the tilings of the entire plane $\Z^2$. Let
$\alpha(x, \cdot)$ be the maximal height function from the set of height functions
with value zero at $x$. 
Denote by $\delta(i, j)$ the expression $i-j~(\textrm{mod}~2)$.
If $x = (a, b)$ and $y = x + (i, j)$, $\alpha(x, y)$ is defined by:
\begin{equation*}
\alpha(x, y) = 
\begin{cases}
2\| y - x \|_{\infty} + \delta(i, j), \qquad \text{ if } \delta(a, b) = 0
\text{ and } i \ge j, \\
2\| y - x \|_{\infty} - \delta(i, j), \qquad \text{ if } \delta(a, b) = 0
\text{ and } i < j, \\
2\| y - x \|_{\infty} - \delta(i, j), \qquad \text{ if } \delta(a, b) = 1
\text{ and } i \ge j, \\
2\| y - x \|_{\infty} + \delta(i, j), \qquad \text{ if } \delta(a, b) = 1
\text{ and } i < j, \\
\end{cases}
\end{equation*}

Alternatively, one can note that $\alpha(x, y)$ is the length of the shortest path from $x$ to $y$ that always has a white square to the left.

Additionally, as in \cite{Pst}, we define a \textit{geodesic path}
$(x_1, x_2, .., x_n)$ by:
\begin{itemize}
\item For $i < n$, the points $x_i$ and $x_{i+1}$ are corners of a common
	$1 \times 1$ square in $\Z^2$;
\item For $i < n$, we have $\| x_{i+1} - x_1 \|_{\infty} = \| x_i - x_1
	\|_{\infty} + 1$.
\end{itemize}

Alternatively, the path
$(x_1, x_2, .., x_n)$ is geodesic if for every $i$, we have
$$\alpha(x_1, x_i) + \alpha(x_i, x_n) = \alpha(x_1, x_n).$$

Furthermore, we let $x \sim_R y$ if there exists a geodesic path between $x$
and $y$ that is entirely within $R$. Finally, we summarize some facts shown in \cite{Pst, Tas}.

\begin{lem}\cite{Pst}\label{pst_lemma}
\begin{enumerate}
\item If $h$ is a valid height function for $R$, then $h(y) - h(x) \le
	\alpha(x, y)$ for all $x, y \in R$ such that $x \sim_R y$.
\item Let $x$ and $y$ be adjacent points in $R$, and suppose there exists a geodesic
	path from $x'$ to $x$ such that the path contains no points on $\partial R$
	other than perhaps $x'$. Then, there exists another geodesic path from $x'$
	to $y$.
\item A region $R$ is tileable by dominoes if and only if
	for every pair of points $x, y \in \partial R$ such that $x \sim_R y$,
        one has	$h(y) - h(x) \le \alpha(x, y)$.
\end{enumerate}
\end{lem}

Using this, we prove the following theorem, which extends
\cite[Lemma 2.2]{Pst} (see also \cite{Tas}):

\begin{thm}\label{extended_kirszbraun}
Let $R$ be a simply connected region in $\Z^2$ such that the point $(0, 0)$ lies
on $\partial R$. Let $H$ be the set of height functions $h$ satisfying
$h(0, 0) = 0$. Then, let $h': U \to \Z$ be defined on the set $U \supseteq
\partial R$ such that $h' = h$ on all points in $\partial R$ and such that $h'
\equiv h \pmod 4$ for all points in $U$, where $h$ is any height
function in $H$. Then, there exists $\hext \in H$ such that $\hext = h'$ for
all points in $U$ if and only if:

\begin{equation}\label{eq:tassy}
    h'(y) - h'(x) \le \alpha(x, y)
\end{equation}
for all pairs of points $x, y \in U$ such that $x \sim_R y$.
\end{thm}
\begin{proof}
The condition is necessary by Fact 1 of Lemma~\ref{pst_lemma}, so we only
need to prove sufficiency. We claim that
\[\hext(y) = \min_{x \in U, x \sim_R y} [h'(x) + \alpha(x, y)]\]
is a valid extension. First, we see that for all points $y \in U$, one has
$\hext(y) = h'(y)$ since $\alpha(y, y) = 0$, and if there were a point $x'$ such
that $x' \sim_R x$ and $h'(x) + \alpha(x, y) \le h'(y)$, the condition \eqref{eq:tassy} would not hold. This implies that the second
condition of Lemma~\ref{fou} is satisfied, as $U \supseteq \partial R$.

For the first condition, note that the value of $h'(x) + \alpha(x, y) \pmod
4$ does not depend on the choice of $x$. Since $h'$ is ``correct'' modulo 4,
$\hext$ satisfies the first condition.

Finally, for the third condition, consider adjacent points $x$ and $y$, and let
$z$ be a point such that $h'(z) + \alpha(z, y) = \hext(y)$. Consider the
geodesic path $(z, ..., x)$ from $z$ to $x$, and let $x'$ be the last point in
the path within $U$. Then, there exists a geodesic path from $x'$ to $x$,
and by the condition of the theorem, $\hext(x) = h'(x') + \alpha(x', x)$. By
Fact 2 of Lemma~\ref{pst_lemma}, this means that there exists a geodesic path
from $x'$ to $y$ and that $\hext(y) \le h'(x') + \alpha(x', y)$. Since $|
\alpha(x', y) - \alpha(x', x) | \le 3$, the condition is satisfied.
\end{proof}

For any point $x \in R$, let $S$ be the largest square centered at $x$ such
that $R \setminus S$ is still tileable with dominoes and let $c(x)$ be the
side length of that square. Then, the following surprisingly characterizes $g(x)
= \hmax(x) - \hmin(x)$ exactly:

\begin{thm}\label{square}
\[g(x) = 2 c(x)\]
\end{thm}
\begin{proof}
Let $g(x) = 4k$ and let $S'$ be the $2k$ by $2k$ square centered at $x$. One
can verify that $S' \subseteq R$, as if two points $a$ and $b$
are adjacent or diagonal to each other on the grid, $|g(a) - g(b)| \le 4$. Let
us denote $\partial S'$ the outer boundary of $S'$. It is trivial to see that
$g(x) \ge 2 c(x)$, as one could always tile $S$ both maximally and minimally 
to get a difference as least as large as $g(x)$. Thus, it remains
to show that $g(x) \le 2c(x)$. 

Consider a possible value for $h(x)$, and notice that if we tile $S'$ maximally
while keeping $h(x)$ at its value, $h(z) = h(x) - \alpha(z, x)$ for all $z \in
\partial S'$. Let $U = \partial R \cup \partial S'$, and observe that if $c(x) <
2k$, there cannot exist an extension to $h$. We note that all values of $h$
on $\partial S'$ must be the ``correct'' value modulo 4, since $h(x)$ must be the
``correct'' value modulo 4. Therefore, by Theorem~\ref{extended_kirszbraun},
there must exist a point $z' \in \partial R$ such that $z' \sim_R z$ and that
either:

\begin{equation*}
h(z) - h(z') > \alpha(z', z)
\end{equation*}

Or

\begin{equation*}
h(z') - h(z) > \alpha(z, z')
\end{equation*}

Let us first consider the first equation. Plugging in our value for $h(z)$, we
see that:

\begin{align*}
h(x) - \alpha(z, x) - h(z') &> \alpha(z', z) \\
h(x) - h(z') &> \alpha(z', z) + \alpha(z, x) \\
h(x) - h(z') &> \alpha(z', x)
\end{align*}

The last statement is true because $\alpha(\cdot, \cdot)$ satisfies the
triangle inequality. Consider the geodesic path from $z'$ to $z$, and let $z''$
be the last point on this path and on $\partial R$. Since we assume that $R$ is
tileable, this means that $h(z) - h(z'') > \alpha(z'', z)$, as $h(z') - h(z'')
\le \alpha(z'', z')$. Next consider the geodesic path from $z''$ to an adjacent
point $y$, which must exist due to the second part of Lemma~\ref{pst_lemma}. Again, let
$z'''$ be the last point on this geodesic path which is on $\partial R$, and
see that again since $R$ is tileable, $h(y) - h(z''') > \alpha(z''', y)$. We
may repeat this process, each time ensuring that $y$ is closer and closer to
$x$, and we see that we eventually get a case where there exists a $z^{\ast}$
such that $z^{\ast} \sim_R x$ and $h(x) - h(z^\ast) > \alpha(z^\ast, x)$. This
would mean a contradiction, as this would mean that our chosen value of $h(x)$
is impossible.

Therefore, the second case must be true, implying that:

\begin{equation}\label{upper}
\begin{aligned}
h(z') - h(x) + \alpha(z, x) &> \alpha(z, z') \\
h(z') - h(x) &> \alpha(z, z') - \alpha(z, x) \\
h(x) &< h(z') - \alpha(z, z') + \alpha(z, x)
\end{aligned}
\end{equation}

Next, consider tiling $S'$ minimally. If $y \in \partial S'$,
then $h(y) = h(x) + \alpha(x, y)$. Similarly to the maximum case, note that
if $R \setminus S'$ is impossible to tile, then there cannot exist an
extension of $h$ in which $S'$ is tiled minimally. Again, by Theorem~\ref{extended_kirszbraun}, this implies that either:

\begin{equation*}
h(y) - h(y') > \alpha(y', y)
\end{equation*}

Or

\begin{equation*}
h(y') - h(y) > \alpha(y, y')
\end{equation*}

Considering the second case, we again get a contradiction by an argument similar to the maximal case:

\begin{align*}
h(y') - h(x) - \alpha(x, y) &> \alpha(y, y') \\
h(y') - h(x) &> \alpha(x, y) + \alpha(y, y') \\
h(y') - h(x) &> \alpha(x, y')
\end{align*}

This means that the second case must be true, meaning the following:

\begin{equation}\label{lower}
\begin{aligned}
h(x) + \alpha(x, y) - h(y') &> \alpha(y', y) \\
h(x) &> h(y') + \alpha(y', y) - \alpha(x, y)
\end{aligned}
\end{equation}

This however, allows us to get the following bound on $g(x)$ by simply
subtracting \eqref{lower} from \eqref{upper}:

\begin{equation*}
\begin{aligned}
g(x) &< h(z') - h(y') - \alpha(z, z') - \alpha(y', y) + \alpha(z, x) +
\alpha(x, y) \\
g(x) &< \alpha(y', z') - \alpha(z, z') - \alpha(y', y) + \alpha(z, x) +
\alpha(x, y) \\
g(x) &< \alpha(y', z) - \alpha(y', y) + \alpha(z, x) + \alpha(x, y) \\
g(x) &< \alpha(y, z) + \alpha(z, x) + \alpha(x, y) \\
g(x) &< \alpha(y, x) + \alpha(x, y) \\
g(x) &< 4k
\end{aligned}
\end{equation*}

This is a contradiction which proves that $S' = S$ and that $g(x) = 2c(x)$.
\end{proof}

\section{Minimum maximum excess}

For a given subset of black squares in $R$, denoted by $U$, let
$N(U)$ denote the neighborhood of $U$ in $R$. The excess
$e(U)$ is defined as $|N(U)| - |U|$. Consider an ordering of the black squares of
$R$, denoted by $b_1, b_2, ..., b_n$. Then, for a positive integer $k \le n$,
let $B_k=\{b_1, b_2, ..., b_k\}$. 
The \textit{maximum excess} of the ordering is the maximum of
$e(B_k)$ when considering all possible values of $k$. Furthermore, define the
\textit{minimum maximum excess} to be the smallest maximum excess over all
possible orderings of the black squares of $R$. Riddle proved the following
link between the minimum maximum excess and the forcing number:

\begin{thm}[\cite{Rid}]\label{riddle}
The forcing number of a region $R$ is bounded below by the minimum maximum
excess over the orderings of the black squares of $R$.
\end{thm}

Lam and Pachter applied this idea to prove the following about squares:

\begin{thm}[\cite{Lp}]\label{lam}
The minimum maximum excess of a $2n \times 2n$ square is equal to~$n$.
\end{thm}

With these two theorems, as well as Theorem~\ref{square}, we can now prove
the main theorem of this paper.

\begin{proof}[Proof of Theorem~\ref{main}]
Let $x$ be the point where $g$ attains its maximum. Set $g(x)=4n$ and let $S$ denote the
$2n$ by $2n$ square centered at $x$, which by Theorem~\ref{square} must be
contained in $R$. For any subset of black squares $U$, let $N_i(U)$ be the set
of all white squares inside $S$ which are adjacent to at least one black square
in $U$. Similarly, let $N_o(U)$ be the set of all white squares outside $S$
that are adjacent to at least one black square in $U$. Notice that because
Theorem~\ref{square} states that $R \setminus S$ must be tileable, any set of
squares outside $S$ must satisfy the Hall Marriage Condition. Specifically, for
any set of black squares $V$ such that $V \cap S = \emptyset$, $|N_o(V)| - |V|
\ge 0$. Now, consider the excess of any given subset $U$ as follows:

\begin{align*}
e(U) &= |N(U)| - |U| \\
&= |N_i(U)| + |N_o(U)| - |U| \\
&= |N_i(U \cap S) \cup N_i(U \setminus S)| + |N_o(U \cap S) \cup N_o(U
\setminus S)| - |U| \\
&\ge |N_i(U \cap S)| + |N_o(U \setminus S)| - |U \cap S| - |U \setminus S| \\
&\ge |N_i(U \cap S)| - |U \cap S| + |N_o(U \setminus S)| - |U \setminus S| \\
&\ge |N_i(U \cap S)| - |U \cap S|.
\end{align*}

This shows that the excess of any given subset $U$ in $R$ is at least the
excess of $U$ restricted to $S$. Next, for any ordering of
the black squares of $R$, consider the sets $B_k$ associated with this
ordering. We see that the $B_k$'s must sequentially contain black squares
inside $S$, forcing some order upon the black squares on $S$. Furthermore, by
the above, the maximum excess of these $B_k$'s must be bounded below by the
maximum excess of the ordering of the squares inside $S$. By Theorem~\ref{lam},
this implies that the maximum excess of any ordering is bounded below by $n$.
This means that the minimum maximum excess of $R$ is at least $n$. This
completes the proof by Theorem~\ref{riddle}.
\end{proof}

\section{Extension to the triangular lattice}

The concept of a height function can be naturally extended to the triangular lattice \cite{Thu}. In this setting, the analog of a domino is a \textit{lozenge} (or diamond), a shape formed by two adjacent triangles. The height function on the triangular lattice can even create an optical illusion, with lozenge tilings perceived as three-dimensional shapes (see Figure~\ref{fig:maxmin}). In this model, we assign black and white colors to the triangles such that all upward-facing triangles are black and all downward-facing triangles are white. For a tiling $T$, we define the height function such that, for every edge $(x, y)$ in a lozenge where a white triangle lies to the left, $h_T(x) - h_T(y) = 1$.

We show that the triangular lattice admits analogs of Lemma~\ref{fou} and Theorem~\ref{extended_kirszbraun}, with the number 4 replaced by 3 throughout.

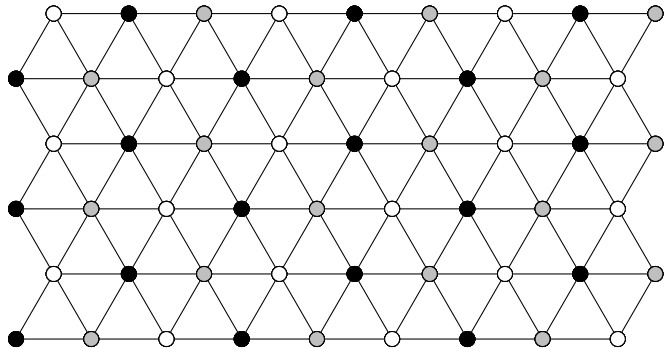
\begin{figure}[ht]
\centering
\begin{tikzpicture}
    \def\blacknode{black}
    \def\whitenode{white}
    \def\graynode{gray!50}
    
    \def\rownum{5}  
    \def\colnum{8}  
    
    \def\xshift{1} 
    \def\yshift{0.866} 
    
    \foreach \row in {0,...,\rownum} {
        \foreach \col in {0,...,\colnum} {
            \pgfmathsetmacro\x{(\col + mod(\row, 2) * 0.5) * \xshift}
            \pgfmathsetmacro\y{\row * \yshift}
            \pgfmathtruncatemacro\rowmod{mod(\row, 2)} 
            
            \ifnum\col<\colnum
                \pgfmathsetmacro\xnext{(\col + 1 + \rowmod * 0.5) * \xshift}
                \draw (\x, \y) -- (\xnext, \y);
            \fi
            
            \ifnum\row<\rownum
                \ifnum\col<\colnum
                    \pgfmathsetmacro\xnext{(\col + 0.5 + 0.5 * \rowmod) * \xshift}
                    \pgfmathsetmacro\ynext{(\row + 1) * \yshift}
                    \draw (\x, \y) -- (\xnext, \ynext);
                \fi
                \ifnum\col=\colnum
                    \ifnum\rowmod=0
                        \pgfmathsetmacro\xnext{(\col + 0.5) * \xshift}
                        \pgfmathsetmacro\ynext{(\row + 1) * \yshift}
                        \draw (\x, \y) -- (\xnext, \ynext);
                    \fi
                \fi
            \fi
            
            \ifnum\row<\rownum
                \ifnum\col>0
                \pgfmathsetmacro\xprev{(\col - 0.5 + 0.5 * mod(\row, 2)) * \xshift}
                \pgfmathsetmacro\ynext{(\row + 1) * \yshift}
                \draw (\x, \y) -- (\xprev, \ynext);
                \fi
                \ifnum\col=0
                \ifnum\rowmod=1
                \pgfmathsetmacro\xprev{(\col - 0.5 + 0.5 * mod(\row, 2)) * \xshift}
                \pgfmathsetmacro\ynext{(\row + 1) * \yshift}
                \draw (\x, \y) -- (\xprev, \ynext);
                \fi
                \fi
            \fi
        }
            \foreach \row in {0,...,\rownum} {
        \foreach \col in {0,...,\colnum} {
            \pgfmathsetmacro\x{(\col + mod(\row, 2) * 0.5) * \xshift}
            \pgfmathsetmacro\y{\row * \yshift}
            \pgfmathtruncatemacro\colorindex{mod(mod(\row, 2) + 2*\col, 3)}
            \ifnum\colorindex=0
                \node[draw, fill=\blacknode, circle, inner sep=2pt] at (\x, \y) {};
            \else
                \ifnum\colorindex=1
                    \node[draw, fill=\whitenode, circle, inner sep=2pt] at (\x, \y) {};
                \else
                    \node[draw, fill=\graynode, circle, inner sep=2pt] at (\x, \y) {};
                \fi
            \fi
        }
    }
    }
\end{tikzpicture}
\caption{Triangular lattice}
\label{fig:triang}
\end{figure}

\begin{lem}\label{lem:triFou}
Let $H_R$ be the set of all height functions corresponding to tilings of $R$ on a triangular grid. Then, a function $h: R \to \Z$ belongs to $H_R$ if and only if the following conditions hold:
\begin{enumerate}
    \item For every $x, y \in R$ such that $x$ and $y$ have the same color in a coloring from Figure~\ref{fig:triang}, we have $h(x) \equiv h(y) \pmod 3$.
    \item For every edge $(x, y) \in \partial R$ such that when crossing from $x$ to $y$ a white triangle is on the left, $h(x) - h(y) = 1$.
    \item For every edge $(x, y) \in R$, $|h(x) - h(y)| \le 2$.
\end{enumerate}
\end{lem}

To adapt Theorem~\ref{extended_kirszbraun} to this lattice, we redefine $\alpha(x, y)$ as the length of the shortest path from $x$ to $y$ that always maintains a white triangle to the left. 
A path $(x_1, x_2, \dots, x_n)$ is said to be geodesic if, for every $i$, the following holds:
$$\alpha(x_1, x_i) + \alpha(x_i, x_n) = \alpha(x_1, x_n).$$
Under these definitions, the statements in Lemma~\ref{pst_lemma} apply as stated.

\begin{thm}\label{thm:triextended_kirszbraun}
Let $R$ be a simply connected region in the triangular lattice such that $(0, 0) \in \partial R$. Let $H$ be the set of height functions $h$ where $h(0, 0) = 0$. Suppose $h': U \to \Z$ is defined on $U \supseteq \partial R$ such that $h' = h$ on $\partial R$ and $h' \equiv h \pmod 3$ for points in $U$, for some $h \in H$. Then, there exists an extension $\hext \in H$ satisfying $\hext = h'$ on $U$ if and only if:
\begin{equation}\label{eq:tritassy}
    h'(y) - h'(x) \le \alpha(x, y)
\end{equation}
for all pairs $x, y \in U$ such that $x \sim_R y$.
\end{thm}

The proof of Theorem~\ref{extended_kirszbraun} holds here with the numbers 4 and 3 replaced by 3 and 2, respectively.

The next theorem, analogous to Theorem~\ref{square}, characterizes height function differences on the triangular lattice.

\begin{thm}
Let $c(x)$ denote the side of the largest regular hexagon $S$ centered at $x$ such that $R \setminus S$ remains tileable with lozenges. Then:
$$\hmax(x) - \hmin(x) = 3c(x).$$
\end{thm}

\begin{proof}
Let $g(x) = \hmax(x) - \hmin(x) = 3k$, and let $S'$ be the regular hexagon with side length $k$ centered at $x$. It follows that $S' \subseteq R$, since $g(x)$ differs by at most 3 between adjacent points. The rest of the proof mirrors that of Theorem~\ref{square}, with all occurrences of 4 replaced by 3. In particular, the final step uses the fact that for any point $y$ on the boundary of $S'$, we have $\alpha(x, y) + \alpha(y, x) = 3k$.
\end{proof}

Theorem~\ref{lam} also generalizes to regular hexagons in the triangular lattice:

\begin{thm}\label{thm:trilam}
The minimum maximum excess of a regular hexagon $R_n$ with side length $n$ is $n$.
\end{thm}

\begin{proof}
This proof adapts arguments from~\cite{Lp}. For a set $U$ of black triangles in $R_n$, there are six transformations, corresponding to shifts along the directions $\theta = \frac{k\pi}{3}$ for $k \in \Z/6\Z$. Each shift transformation preserves or increases $|N(U)|$, the count of white triangles bordering at least one black triangle. For any direction $\theta$ of the form $\frac{k\pi}{3}$ and any subset $U \subseteq V$, we have $f_\theta(U) \subseteq f_\theta(V)$. Now consider the first set $B_k$ such that $f_0(f_\frac{-\pi}{3}(B_k))$ spans every row among the $n$ bottom horizontal rows of $R_n$. Such a $k$ exists because $f_0(f_\frac{-\pi}{3}(B_k))$ grows one triangle at a time. Also, if a row is non-empty, every row below it must also be non-empty. Therefore, $B_k$ has at least one triangle less in each of the $n$ bottom rows than $N(B_k)$, ensuring that its excess is at least $n$.
\end{proof}

\begin{proof}[Proof of Theorem~\ref{thm:trimain}]
Let $x$ be the point where $g$ attains its maximum, with $g(x) = 3n$. Let $S$ be the regular hexagon centered at $x$ with side length $n$. By Theorem~\ref{thm:trilam}, $S$ is contained in $R$, and $R \setminus S$ is tileable by lozenges. The remainder of the proof proceeds as in Theorem~\ref{main}.
\end{proof}

Now we can calculate the forcing numbers for hexagons on a triangular lattice. The sides of the hexagon in the counterclockwise order should have the form $a$, $b$, $c$, $a+t$, $b-t$ and $c+t$. Otherwise, the boundary will not close. When one goes around the boundary, the total change in height function is equal to $a-b+c-(a+t)+(b-t)-(c+t) = 3t$. Thus, a hexagon lacking central symmetry cannot support a valid height function at the boundary, making it non-tileable. Thus, we restrict our focus to centrally symmetric hexagons, which are described by three natural numbers $a$, $b$ and $c$, representing side lengths in three directions.

\begin{figure}[ht]
\centering
\usetikzlibrary{calc}
\newcommand*\rowsa{3} 
\newcommand*\rowsb{4}  
\newcommand*\rowsc{6} 

\begin{tikzpicture}[scale=0.75]
    \foreach \row in {0, ..., 2} {
        \foreach \col in {0, ..., 5} {
            \draw ($(0.5 + \col, 0.866 * 7) + \row*(0.5, -0.866)$) --++ (1, 0) --++ (0.5, -0.866) --++ (-1, 0) --++ (-0.5, 0.866);
        }
    }
    \node[label={[label distance=-3mm]30:$x$}] at (2, 3.464) {};

    \foreach \col in {0, ..., 2} {
        \foreach \row in {0, ..., 3} {
            \draw ($(0, 0.866 * 6) - \row*(0.5, 0.866) + \col*(0.5, -0.866)$) --++ (0.5, -0.866) --++ (0.5, 0.866) --++ (-0.5, 0.866) --++ (-0.5, -0.866);  
        }
    }

    \foreach \row in {0, ..., 3} {
        \foreach \col in {0, ..., 5} {
            \draw ($(0, 0) + \row*(0.5, 0.866) + \col*(1, 0)$) --++ (1, 0) --++ (0.5, 0.866) --++ (-1, 0) --++ (-0.5, -0.866);
        }
    }
    
    \foreach \row in {0, ..., 3} {
        \foreach \col in {0, ..., 5} {
            \draw ($(11.5 + \col, 0.866 * 7) + \row*(-0.5, -0.866)$) --++ (1, 0) --++ (-0.5, -0.866) --++ (-1, 0) --++ (0.5, 0.866);
        }
    }

    \foreach \col in {0, ..., 2} {
        \foreach \row in {0, ..., 3} {
            \draw ($(17, 0.866 * 6) - \row*(0.5, 0.866) + \col*(0.5, -0.866)$) --++ (0.5, -0.866) --++ (0.5, 0.866) --++ (-0.5, 0.866) --++ (-0.5, -0.866);
        }
    }

    \foreach \row in {0, ..., 2} {
        \foreach \col in {0, ..., 5} {
            \draw ($(9.5 + \col, 0.866 * 3) + \row*(0.5, -0.866)$) --++ (1, 0) --++ (0.5, -0.866) --++ (-1, 0) --++ (-0.5, 0.866);
        }
    }
\end{tikzpicture}
    \caption{Minimal and maximal tilings of a hexagon with sides $3$, $4$, $6$ with lozenges.}
    \label{fig:maxmin}
\end{figure}
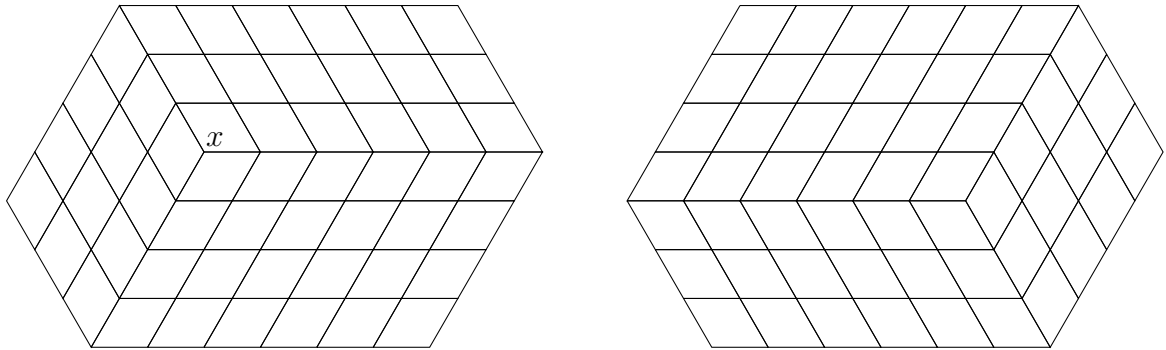

\begin{thm}\label{thm:forcehex}
For a centrally symmetric hexagon on the triangular lattice with sides $a$, $b$ and $c$, the forcing number is $\min(a, b, c)$.
\end{thm}

\begin{proof}
Without loss of generality, assume $a \le b \le c$ and consider the maximal and minimal tilings of $R$, as depicted in Figure~\ref{fig:maxmin}. To establish the upper bound of $a = \min(a, b, c)$, observe the forcing set of size $a$, originating diagonally from the angle between the sides of lengths $b$ and $c$, as illustrated in Figure~\ref{fig:forcing_hexagon}. This forcing set is part of the minimal tiling and uniquely determines it.

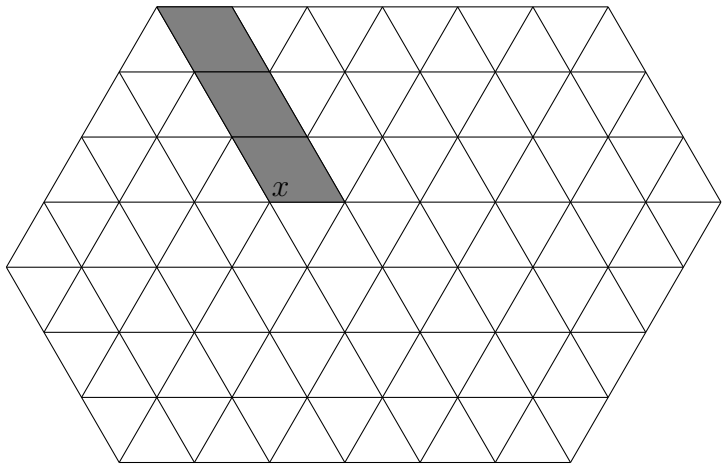
\begin{figure}[ht]
\centering
\usetikzlibrary{calc}
\newcommand*\rowsa{3} 
\newcommand*\rowsb{4}  
\newcommand*\rowsc{6} 

\begin{tikzpicture}
    \foreach \row in {0, 1, ..., \numexpr\rowsa+\rowsb} {
        \ifnum\row<\rowsa
            \coordinate (start) at ($\row*(-0.5, 0.866)$);
        \else
            \coordinate (start) at ($\rowsa*(-0.5, 0.866) + {\row-\rowsa}*(0.5, 0.866)$);
        \fi
        \ifnum\row<\rowsb
            \coordinate (end) at ($(\rowsc, 0) + \row*(0.5, 0.866)$);
        \else
            \coordinate (end) at ($(\rowsc, 0) + \rowsb*(0.5, 0.866) + {\row-\rowsb}*(-0.5, 0.866)$);
        \fi
        \draw (start) -- (end);
    }

    \foreach \col in {0, 1, ..., \numexpr\rowsb+\rowsc} {
        \ifnum\col<\rowsb
            \coordinate (start) at ($\rowsa*(-0.5, 0.866)+\col*(0.5, 0.866)$);
        \else
            \coordinate (start) at ($\rowsa*(-0.5, 0.866)+\rowsb*(0.5, 0.866) + {\col-\rowsb}*(1, 0)$);
        \fi
        \ifnum\col<\rowsc
            \coordinate (end) at ($\col*(1, 0)$);
        \else
            \coordinate (end) at ($\rowsc*(1, 0) + {\col-\rowsc}*(0.5, 0.866)$);
        \fi
        \draw (start) -- (end);
    }

    \foreach \col in {0, 1, ..., \numexpr\rowsa+\rowsc} {
        \ifnum\col<\rowsc
            \coordinate (start) at ($\rowsc*(1, 0)+\col*(-1, 0)$);
        \else
            \coordinate (start) at (${\col-\rowsc}*(-0.5, 0.866)$);
        \fi
        \ifnum\col<\rowsa
            \coordinate (end) at ($\rowsc*(1, 0) + \rowsb*(0.5, 0.866) + \col*(-0.5, 0.866)$);
        \else
            \coordinate (end) at ($\rowsc*(1, 0) + \rowsb*(0.5, 0.866) + \rowsa*(-0.5, 0.866) + {\col-\rowsa}*(-1, 0)$);
        \fi
        \draw (start) -- (end);
    }

    \draw[fill=gray] (0.5, 0.866 * 7) --++ (1, 0) --++ (0.5, -0.866) --++ (-1, 0) --++ (-0.5, 0.866);
    \draw[fill=gray] (1, 0.866 * 6) --++ (1, 0) --++ (0.5, -0.866) --++ (-1, 0) --++ (-0.5, 0.866);
    \draw[fill=gray] (1.5, 0.866 * 5) --++ (1, 0) --++ (0.5, -0.866) --++ (-1, 0) --++ (-0.5, 0.866);
    \node[label={[label distance=-3mm]30:$x$}] at (2, 3.464) {};
\end{tikzpicture}
    \caption{Forcing set for a hexagon with sides $3$, $4$, $6$.}
    \label{fig:forcing_hexagon}
\end{figure}

The lower bound follows from Theorem~\ref{thm:trimain}. Let $x$ be the point on the bisector of the angle between the sides of lengths $b$ and $c$, positioned at a distance $a$ from the angle. It is easy to see that $g(x) = 3a$ and so the forcing number equals $a = \min(a, b, c)$.
\end{proof}

\section*{Acknowledgments}

This paper originated from discussions between the authors at the Olga Radko 
Math Circle --- a Sunday math school for gifted children. Determining the forcing number of a $6 \times 6$ square was 
one of the problems featured in a handout on tilings, taken from the 2012 
Olimpiada Nacional Escolar de Matemática in Peru. The authors express their 
gratitude to the Math Circle for the valuable insights and learning experiences
it provided.

We are also thankful to Professor Igor Pak for many useful remarks.

\end{document}